\documentclass[12pt,twoside,reqno]{amsart}

\usepackage[OT1]{fontenc}
\usepackage{type1cm}

\usepackage{comment}
\usepackage{enumerate}
\usepackage{mathabx}

\usepackage{amsthm}
\RequirePackage{amsmath,amsfonts,amssymb,amsthm}

\RequirePackage[dvips]{graphicx}

\usepackage{psfrag}
\usepackage[usenames]{color}
  \numberwithin{equation}{section}

\usepackage[active]{srcltx}  

  \newcommand{\N}{\mathbb{N}}         
  \newcommand{\Z}{\mathbb{Z}}         
  \newcommand{\R}{\mathbb{R}}         
  \newcommand{\C}{\mathbb{C}}         
  \newcommand{\D}{\mathbb{D}}         
  \newcommand{\CmR}{\mathbb{C}\setminus\mathbb{R}}
  \newcommand{\DmR}{\mathbb{D}\setminus\mathbb{R}}
  \newcommand{\PP}{\mathbb{P}}         
  \newcommand{\bba}{\mathbf{a}}

\def\Dk{{\mathcal D}}
\def\Ak{{\mathcal A}}
\def\Bk{{\mathcal B}}
\def\Ek{{\mathcal E}}
\def\Nk{{\mathcal N}}

  \newcommand{\e}{\varepsilon}

  \newcommand{\lam}{\lambda}

  \newcommand{\ov}{\overline}
  \newcommand{\eps}{\varepsilon}
  \newcommand{\wtil}{\widetilde}
  \newcommand{\be}{\begin{equation}}
  \newcommand{\ee}{\end{equation}}
  \newcommand{\ba}{\begin{align}}
  \newcommand{\ea}{\end{align}}
  \newcommand{\bas}{\begin{align*}}
  \newcommand{\eas}{\end{align*}}
  \def\bx{{\mathbf x}}
  
  \def\half{\frac{1}{2}}

  \newtheorem{thm}{Theorem}[section]
  \newtheorem{mainthm}{Theorem}
  
  \newtheorem{lemma}[thm]{Lemma}
  \newtheorem{prop}[thm]{Proposition}
  \newtheorem{cor}[thm]{Corollary}

  \theoremstyle{remark}
  
  \newtheorem{rems}[thm]{Remarks}

\begin{document}

\title{Absolute continuity of complex Bernoulli convolutions}

\author{Pablo Shmerkin}
\address{Department of Mathematics and Statistics, Torcuato Di Tella University, and CONICET, Buenos Aires, Argentina}
\email{pshmerkin@utdt.edu}
\thanks{P.S. was partially supported by Projects PICT 2011-0436 and PICT 2013-1393 (ANPCyT)}
\urladdr{http://www.utdt.edu/profesores/pshmerkin}

\author{Boris Solomyak}
\address{University of Washington and Bar-Ilan Univeristy}
\email{solomyak@math.washington.edu}
\thanks{B.S. has been supported in part by  NSF grants DMS-0968879 and DMS-1361424.}
\urladdr{http://www.math.washington.edu/~solomyak/personal.html}


\begin{abstract}
We prove that  complex Bernoulli convolutions are absolutely continuous in the supercritical parameter region, outside of an exceptional set of parameters of zero Hausdorff dimension. Similar results are also obtained in the biased case, and for other parametrized families of self-similar sets and measures in the complex plane, extending earlier results.
\end{abstract}

\maketitle

\section{Introduction and statement of results}

Recall that, given $\lam\in (1/2,1)$, the \emph{Bernoulli convolution} $\nu_\lam$ is the distribution of the random sum $\sum_{n=1}^\infty \pm \lam^n$, where the signs are chosen with equal probability. The study of this family has a long history, dating back to Erd\H{o}s' seminal papers \cite{Erdos39, Erdos40}. The most important problem around Bernoulli convolutions is to determine for which values of $\lam$ it is absolutely continuous (and when it is, find out what can be said about its density). The only known values for which $\nu_\lam$ is singular are reciprocals of Pisot numbers in $(1,2)$. In the opposite direction, recently we have shown that, outside of a set of $\lambda$ of zero Hausdorff dimension, $\nu_\lam$ is absolutely continuous, and has a density in $L^q$ for some $q=q(\lam)>1$. See \cite{Shmerkin13, ShmerkinSolomyak14} for the proofs, and further background and references on Bernoulli convolutions.

Bernoulli convolutions have an immediate generalization to the complex plane: if $\lambda$ is now a non-zero complex number in the open unit disk $\mathbb{D}$, then we can still define a measure $\nu_\lam$ as the distribution of the random sum $\pm\lam^n$, with the signs chosen independently with equal probability. Denote the support of $\nu_\lam$ by $A_\lam$; this is a compact set satisfying the self-similarity relation $A_\lam = (\lam A_\lam-1)\cup (\lam A_\lam+1)$. When $|\lambda|<1/\sqrt{2}$, one has $\dim_H(A_\lam)\le \frac{\log 2}{|\log\lam|}<2$, and hence $\nu_\lam$ is necessarily singular. We remark that it is far from clear how to determine the measure and topology of $A_\lam$ when $|\lam|>1/\sqrt{2}$. This is in contrast to real Bernoulli convolutions: for $\lam\in (1/2,1)$, it is an easy fact that the support of $\nu_\lam$ is an interval.

Most of the previous work on this family concerns the structure of its \emph{connectedness locus} (sometimes also called the \emph{Mandelbrot set for pairs of linear maps}) $\mathcal{M}=\{\lam\in \D: A_\lam \text{ is connected}\}$. The study of this set (and of the family $A_\lam$) was pioneered by  Barnsley and Harrington \cite{BarnsleyHarrington85}. Bandt \cite{Bandt02} conjectured that, away from the real axis, $\mathcal{M}$ is the closure of its interior; this was recently proved by Calegari, Koch and Walker \cite{CKW14}. See \cite{CKW14} and references there for more on the properties of $\mathcal{M}$.

In this article, motivated by the study of real Bernoulli convolutions, we investigate a different problem: what can we say about the set  $\{ \lam\in \mathbb{D}:|\lam|\in (2^{-1/2},1),\nu_\lam \text{ is singular}\}$? Away from this set, what can we say about the density of $\nu_\lam$? Note that answers to these questions have immediate consequences for $A_\lambda$. This problem was first addressed in \cite{SolomyakXu03} (although similar problems for the sets $A_\lam$ were investigated earlier in \cite{PeresSchlag00}), where the following is proved:

\begin{thm}[{\cite[Theorem 2.10]{SolomyakXu03}}]
Let
\begin{align*}
U_1 &= \{ \lam\in\CmR: 2^{-1/2} < |\lam| < 2\cdot 5^{-5/8}\},\\
U_2 &= \{ \lam\in\CmR: \lambda^k\in U_1 \text{ for some } k\ge 2\}.
\end{align*}
Then $\nu_\lam$ is absolutely continuous with an $L^2$ density for almost all $\lam\in U_1$, and $\nu_\lam$ is absolutely continuous with a continuous density for almost all $\lam\in U_2$.

In particular, $A_\lam$ has positive Lebesgue measure for almost all $\lam\in U_1$ and has nonempty interior for almost all $\lam\in U_2$.
\end{thm}

Based on the work of Peres and Schlag \cite{PeresSchlag00}, the paper \cite{SolomyakXu03} also gave bounds on the Hausdorff dimension on the sets $\{ \lam\in U'_1: \,\nu_\lam\notin L^2\}$ for compact sub-regions $U'_1\subset U_1$; these bounds are much larger than zero. (Here, and below, by $\mu\in\mathcal{F}$, where $\mathcal{F}$ is a function space, we mean that $\mu$ is absolutely continuous, and its Radon-Nikodym derivative is in $\mathcal{F}$.) We emphasize that the region $U_1\cup U_2$ is far from covering all of the supercritical parameter region
\[
U:= \{ \lambda\in\C\setminus\R: |\lambda|\in (2^{-1/2},1) \},
\]
so even almost sure type of results were lacking here (by contrast, it has been known since \cite{Solomyak95} that real Bernoulli convolutions are absolutely continuous for almost all $\lam\in (1/2,1)$). The following is our first main result.

\begin{mainthm} \label{mainthm:complex-BC}
There exists a set $E\subset U$ with $\dim_H E=0$, such that $\nu_\lam$ is absolutely continuous and has a density in $L^q$ for some $q=q(\lambda)>1$. In particular, $A_\lam$ has positive Lebesgue measure for all $\lam\in U\setminus E$.
\end{mainthm}

\begin{rems}
\begin{enumerate}
\item We do not know if the density is in $L^2$ (or even in $L^q$ for some fixed $q>1$) outside of a zero dimensional set of exceptions. We do get some new information for $|\lam|$ very close to $1$ and away from the real axis; see Theorem \ref{mainthm:E-K}(iii) below.
\item Certainly one cannot hope to get absolute continuity for \emph{all} $\lam\in\mathbb{D}\setminus \R$ such that $s(\lam,p)>2$: it is shown in \cite[Theorem 2.3]{SolomyakXu03} that $\nu_\lam$ is singular whenever $\lam^{-1}$ is a so-called complex Pisot number, that is, $\lam^{-1}$ is algebraic and all of its algebraic conjugates, other than its complex conjugate, lie in $\D$. There are infinitely many complex Pisot numbers $\lam$ with $\lam^{-1}\in U$.
\item Recently Hare and Sidorov \cite{HS} proved that $A_\lam$ has nonempty interior {\em for all} nonreal $\lam$ with $|\lam| \in (2^{-1/4},1)$. However, their results do not say anything about $\nu_\lam$.
\end{enumerate}
\end{rems}

In fact, we obtain a more general result than Theorem \ref{mainthm:complex-BC}. Firstly, our results also hold in the biased case, that is, when the signs $\pm$ are chosen with different probabilities $p, 1-p$; moreover, the exceptional set is independent of the bias, modulo the fact that the supercritical parameter region changes. Secondly, we obtain an analogous result for an arbitrary choice of translation vectors.

In order to state the result, we introduce some notation. For a finite index set $\Lambda$, let $\PP_\Lambda$ be the open simplex of probability vectors $p=(p_i)_{i\in\Lambda}$ with $p_i>0$. Given $p\in\PP_\Lambda$, $\mathbf{a}=(a_i)_{i\in\Lambda}\in\C^\Lambda$, and $\lam\in\D$, let $\nu_{\lam,\mathbf{a}}^p$ be the self-similar measure corresponding to the IFS  $(f_i: z\mapsto \lam z+a_i)_{i\in\Lambda}$ with weights $(p_i)_{i\in\Lambda}$, that is, the only Borel probability measure satisfying the relation
\[
\nu_{\lam,\mathbf{a}}^p = \sum_{i\in\Lambda} p_i \, f_i\nu_{\lam,\mathbf{a}}^p.
\]
Here and below, if $\mu$ is a measure on $X$ and $g:X\to Y$ is a map, then $g\mu(A)=\mu(g^{-1}A)$ is the push-forward measure. Further, let $s(\lam,p)$ be the similarity dimension of this measure; explicitly
\[
s(\lam,p) = \frac{h(p)}{-\log(\lam)},
\]
where $h(p)=-\sum_{i\in\Lambda} p_i \log(p_i)$ is the entropy of $p$. Recall that the (lower) Hausdorff dimension of a measure $\mu$ is
\[
\dim\mu = \inf\{\dim_H A: \mu(A)>0\}.
\]
It is well known that $\dim\nu_{\lam,\mathbf{a}}^p\le s(\lam,p)$, and in particular $\nu_{\lam,\mathbf{a}}^p$ is singular whenever $s(\lam,p)<2$ (note that the translations do not come up in $s(\lam,p)$). Further, let $A_{\lam,\mathbf{a}}$ be the attractor of $( z\mapsto \lam z+a_i)_{i\in\Lambda}$ or, alternatively, the topological support of $\nu_{\lam,\mathbf{a}}^p$.

\begin{mainthm} \label{mainthm:complex-SSM}
Fix $\mathbf{a}=(a_i)_{i\in\Lambda}$, where $m=\#\Lambda\ge 2$ and all the $a_i$ different. Then there is a set $E\subset \mathbb{D}\setminus\R$ of zero Hausdorff dimension, such that for any $\lam\in\mathbb{D}\setminus (E\cup \R)$ and any $p\in\PP_\Lambda$ such that $s(\lam,a,p)>2$, the measure $\nu_{\lam,\mathbf{a}}^p$ is absolutely continuous, and has a density in $L^q$ for some $q=q(\lam)>1$.

In particular, if $\lam\in \D\setminus (E\cup\R)$ and $\lam^2 m>2$, then $A_{\lam,\mathbf{a}}$ has positive Lebesgue measure.
\end{mainthm}

We make some remarks on this statement.
\begin{rems}
\begin{enumerate}
\item Note that $\nu_\lam = \nu_{\lam,(-1,1)}^{(1/2,1/2)}$, and $s(\lam,(\tfrac12,\tfrac12))=\log 2/|\log\lam|$. Hence Theorem \ref{mainthm:complex-BC} is immediate from Theorem \ref{mainthm:complex-SSM}.
\item This is a direct analog of the results on real self-similar measures from \cite{Shmerkin13, ShmerkinSolomyak14}, and we use the same scheme of proof. We note however that in the complex case the family $\{\nu_{\lam,\mathbf{a}}^p\}_{\lam\in\D}$ is two-dimensional, and we still get a zero-dimensional set of exceptions (by contrast, the bound on the dimension of the exceptional set in e.g. \cite[Theorem A]{ShmerkinSolomyak14} or \cite[Theorem 1.10]{Hochman15} is equal to the dimension of the parameter space minus one).
\item When $|\Lambda|=2$, the case $\lam\in\R$ reduces back to the family of real Bernoulli convolutions. However, if $|\Lambda|\ge 3$ and the vectors $a_i$ are not collinear, then the exclusion of $\lam\in\R$ may appear artificial. The main reason behind this is that, because we are interested in the regime where we expect the measure to be absolutely continuous, we cannot hope to rule out the third alternative (saturation along lines) in \cite[Theorem 1.5]{Hochman15}, and this causes our  current proof to break down (however, see Section \ref{subsec:gasket} for a special case where we do get information for real $\lam$).
\end{enumerate}
\end{rems}

Our method applies equally well to other parametrized families of self-similar measures on the complex plane. For concreteness, we focus on one example, concerning arithmetic sums of similar copies of $C_\lambda$. Recall that given two sets $A,B\subset\R^d$, their arithmetic sum is $A+B=\{x+y:x\in A, y\in B\}$. The problem of estimating the dimension, measure and topology of arithmetic sums of Cantor sets has received a great deal attention in the last decades, motivated in part by conjectures of Furstenberg about sums of Cantor sets with certain arithmetic structure (see \cite{PeresShmerkin09, HochmanShmerkin12}) and of Palis regarding sums of Cantor sets arising in smooth dynamics, see \cite{MoreiraYoccoz01}. The vast majority of results, however, are about sums of subsets of $\R$.

The analog of arithmetic sums for measures is convolution. Recall that the convolution of two finite measures $\mu,\nu$ on $\R^d$ is the push-down of the product $\mu\times\nu$ under the map $(x,y)\mapsto x+y$. The Lebesgue measure in $\R^d$ is denoted $\mathcal{L}$.
We have the following result:

\begin{mainthm} \label{mainthm:convolutions}
Suppose that $\mathbf{a}=(a_i)_{i\in\Lambda}$ and $\lam\in\D\setminus \R$ are such that $(\lam z+ a_i)_{i\in\Lambda}$ satisfies the open set condition. There is a set $E\subset\C$ of zero Hausdorff dimension such that the following holds: for every $p\in\PP_\Lambda$ such that $s(\lam,p)>1$ and every $u\in\C\setminus E$,
\[
\nu_{\lam,\mathbf{a}}^p * S_u \nu_{\lam,\mathbf{a}}^p \ll \mathcal{L},
\]
where $S_u(z)=u z$, and moreover the density is in $L^q$ for some $q=q(u)>1$. In particular, if $\lam>1/\#\Lambda$, then $\mathcal{L}(A_{\lam,\mathbf{a}}+u A_{\lam,\mathbf{a}})>0$ for all $u\in\C\setminus E$.
\end{mainthm}

\begin{rems}
\begin{enumerate}
\item It will emerge from the proof that the open set condition can be weakened substantially, to no super-exponential concentration of cylinders (see Section \ref{subsec:hochman}). It is clear that some assumption is needed, though, as if there was a dimension drop already for  $\nu_{\lam,\mathbf{a}}^p$, it would carry over to the convolution.
\item In general, the claim does not hold for all $u$, and not even for $u=1$, since $A_\lambda+A_\lambda$ is the attractor of $(\lambda z,\lambda z+1,\lambda z+2)$ which has only three maps (there is an exact overlap), so $\dim_H A_\lambda>1$ but $\dim_H(A_\lam+A_\lam)<2$ for $|\lambda|\in (1/2,1/\sqrt{3})$.
\end{enumerate}
\end{rems}

Following the ideas from \cite{Shmerkin13, ShmerkinSolomyak14}, the strategy to prove Theorems \ref{mainthm:complex-SSM} and \ref{mainthm:convolutions} is to decompose the measures in question as a convolution of two measures, the first of which has full dimension outside of a zero dimensional set of exceptional parameters (this relies on deep recent results of Hochman), while the second measure has power Fourier decay (again outside of a small parameter set). This
power decay is achieved  by adapting what has come to be known as the Erd\H{o}s-Kahane argument. As a direct consequence, we obtain that for $\lam$ close enough to $1$ in modulus, $\nu_{\lam,\mathbf{a}}^p$ has a $C^k$ density outside of a set of arbitrarily small dimension -- this is what Erd\H{o}s and Kahane proved in the real case.

\begin{mainthm} \label{mainthm:E-K}
\begin{enumerate}
\item[\rm{(i)}] There is a set $E\subset\D\setminus\R$ of zero Hausdorff dimension, such that if $\mathbf{a}=(a_i)_{i\in\Lambda}$ with  not all of the $a_i$ equal, $p\in\mathbb{P}_\Lambda$, and $\lam\in\D\setminus (E\cup\R)$, then there are $C,\gamma>0$ (depending on $\lam,\mathbf{a},p$), such that
    \be \label{eq:power-decay}
    |\widehat{\nu_{\lam,\mathbf{a}}^p}(\xi)| \le C\, |\xi|^{-\gamma} \quad\text{for all }\xi\in\C\setminus\{0\}.
    \ee
    More precisely, fix $\e>0$ and a region $H=H_{b_1,b_2,\eta} := \{ z\in\C: b_1\le |z|\le b_2, \Im(z)>\eta\}$ with $1<b_1<b_2$ and $\eta>0$. Then there exist $\gamma,C>0$ and a set $\mathcal{E}\subset H $ with $\dim_H(\mathcal{E})<\e$, such that \eqref{eq:power-decay} holds for all $\lam$ such that $\lam^{-1}\in H \setminus \mathcal{E}$.
\item[\rm{(ii)}] Given $\lam\in\C\setminus\R$,  there is a set $E\subset\C$ of zero Hausdorff dimension, such that whenever $\#\Lambda\ge 3$, $p\in\mathbb{P}_\Lambda$ and $\mathbf{a}=(a_i)_{i\in\Lambda}$ satisfies $(a_k-a_i)/(a_j-a_i)\notin E$ for some distinct $i,j,k\in\Lambda$, there are $C,\gamma>0$ (depending on $\lam,\mathbf{a},p$), such that
    \[
    |\widehat{\nu_{\lam,\mathbf{a}}^p}(\xi)| \le C\, |\xi|^{-\gamma} \quad\text{for all }\xi\in\C\setminus\{0\}.
    \]
\item[\rm{(iii)}] For any $k\in\N$, $\eta>0$, $p\in\mathbb{P}_\Lambda$, and $\mathbf{a}=(a_i)_{i\in\Lambda}$ with all $a_i$ different,
\[
\lim_{\delta\to 0}\dim_H(\{\lam\in\C\setminus\R: |\lambda|\in (1-\delta,1), |\Im(\lam)|>\eta \text{ and } \nu_{\lam,\mathbf{a}}^p\notin C^k\})=0.
\]
\end{enumerate}
\end{mainthm}

Some applications and variants of our main results will be briefly discussed in Section \ref{sec:generalizations}.

\section{Power Fourier decay: Proof of Theorem \ref{mainthm:E-K} }

We start by proving the first part of Theorem \ref{mainthm:E-K}, yielding power Fourier decay of $\nu_{\lam,\bba}^p$ for non-real $\lam$ outside of a zero dimensional set.

\begin{proof}[Proof of Theorem \ref{mainthm:E-K}(i)]
Let $\lam$ be a complex number of modulus less than 1, non-real. Suppose for simplicity that $\Lambda=\{1,\ldots,m\}$ with $a_1\neq a_2$. Since replacing $a_i$ by $(a_i-a_1)/(a_2-a_1)$ has the effect of applying a linear map to the measures in question, we may assume that $a_1=0$ and $a_2=1$.  The definition of $\nu_{\lam,\bba}^p$ as a self-similar measure easily yields
\be \label{eq:FT-product}
\widehat{\nu_{\lam,\bba}^p}(\xi) = \int_\C e^{2\pi i \Re(z\ov{\xi})}\,d\nu_{\lam,\bba}^p(z) = \prod_{n=0}^\infty \sum_{j=1}^m p_j \exp[2\pi i \Re(\lam^n a_j \ov{\xi})].
\ee
Then
\begin{eqnarray*}
|\widehat{\nu_{\lam,\bba}^p}(\xi)| & \le & \prod_{n=0}^\infty \Bigl(\bigl|p_1 + p_2 e^{2\pi i \Re(\lam^n \ov{\xi})}\bigr| + (1-p_1-p_2)\Bigr) \\
                                & \le & \prod_{n=0}^\infty \bigl(1 - c_1\|\Re(\lam^n \ov{\xi})\|^2\bigr),
\end{eqnarray*}
for some $c_1>0$ depending only on $p$. Here and below $\|x\|$ denotes the distance from $x$ to the nearest integer.
Let $\xi = \ov{t \lam^{-N}}$ with $|t| \in [1, |\lam|^{-1}]$. Then we have, denoting $\theta=\lam^{-1}$:
\be \label{eq:EKdecay}
|\widehat{\nu_{\lam,\bba}^{p}}(\xi)| \le  \prod_{n=1}^N (1-c_1\|\Re(\theta^n t)\|^2).
\ee
Therefore, the desired power decay of the Fourier transform will follow if   $\|\Re(\theta^n t)\|$ is bounded away from zero
for $n$ in a subset of positive lower density, uniformly in $t$, satisfying $|t|\in [1,|\theta|]$.

Given $1< b_1 < b_2<\infty$ and $\eta>0$, recall that
\be \label{defH}
H_{b_1,b_2,\eta} = \{ z \in \C:\ b_1 \le  |z| \le  b_2,\ \Im(z) > \eta\}.
\ee
Clearly, it is enough to prove the  claim concerning  $\{\lam:\ \lam^{-1}\in H_{b_1,b_2,\eta}\}$ for
all $b_1, b_2,\eta$ (by symmetry, we can assume that $\lam$ is in the upper half-plane). Thus, we will fix $b_1,b_2,\eta$ below.

\begin{prop} \label{prop:EK-combinatorial}
There is a constant $\rho=\rho(b_1,b_2,\eta)>0$ such that the set
\[\Ek_{N,\delta}:=
\left\{ \theta\in H_{b_1,b_2,\eta}:\, \frac{1}{N} \min_{|t|\in [1,|\theta|]} \#\{n\in \{1,\ldots,N\}:\, \|\Re(\theta^n t)\|\le \rho\} > 1- \delta\right\}
\]
can be covered by $\exp(C_2\delta\log(1/\delta) N)$ balls of radius $b_1^{-N}$, where $C_2=C_2(b_1,b_2,\eta)$ is independent of $N$.
\end{prop}

Let us derive Theorem~\ref{mainthm:E-K}(i) first. Consider the set
$$
\Ek_{\delta} = \limsup_N \Ek_{N,\delta} = \bigcap_{k=1}^\infty \bigcup_{N = k}^\infty \Ek_{N,\delta}.
$$
By Proposition~\ref{prop:EK-combinatorial}, $\dim_H(\Ek_{\delta} ) \le \frac{C_2 \delta\log(1/\delta)}{\log b_1}$. On the other hand, for
$\lam^{-1}=\theta\in H_{b_1,b_2,\eta} \setminus \,\Ek_{\delta}$ we have for all $N\ge N_0$:
$$
|\widehat{\nu_{\lam,\bba}^{p}}(\xi)| \le (1-c_1\rho^2)^{\delta N}\ \ \mbox{for}\ \ |\xi| \in [|\theta|^N, |\theta|^{N+1}],
$$
hence
$$
|\widehat{\nu_{\lam,\bba}^{p}}(\xi)| \le (b_2|\xi|)^{\delta \log(1-c_1\rho^2)/\log b_2}\ \ \mbox{for all}\ \ |\xi| \ge b_1^{N_0}.
$$
Since $\dim_H(\Ek_\delta)\to 0$ as $\delta\to 0$, this yields the desired power Fourier decay outside of a set of small dimension.
 \end{proof}

\medskip

\begin{proof}[Proof of Proposition~\ref{prop:EK-combinatorial}]
Following the  scheme of the Erd\H{o}s-Kahane's argument, let
\be \label{eq:K+eps}
\Re(\theta^n t) = K_n +\eps_n,\ \ \ \mbox{where}\ \ K_n \in \Z, \ |\eps_n|\le 1/2.
\ee

In the next technical lemma we show how to recover $\theta$ and $z_0$ (under appropriate conditions),
given $x_j = \Re(\theta^j z_0),\ j=0,1,2,3$.
Recall that $b_1,b_2,\eta$ are fixed and for $R_0>0$ let
$$
V_{R_0} =\left\{\bx = (x_0,x_1,x_2,x_3):\ x_j = \Re(\theta^j z_0),\ 0\le j\le 3,\ |z_0|\ge R_0,\ \theta\in H_{b_1,b_2,\eta}\right\}.
$$
Denote by $\Nk_\eps (V)$ the $\eps$-neighborhood of $V$ in the $\ell^\infty$ metric.
\begin{lemma} \label{lem-tech}
There exist $R_0>0$ and $C_3>0$ depending only on $b_1,b_2,\eta$ such that
there are continuously differentiable functions $$F:\,\Nk_1(V_{R_0})\to \{\theta:\ |\theta|>1, \ \Im(\theta)>0\}\ \ \mbox{and}\ \  G:\,\Nk_1(V_{R_0})\to \R,$$ such that $\theta = F(\bx)$ and $y_3 = G(\bx)$ satisfy
\be \label{equs}
x_j = \Re(\theta^{j-3} (x_3+iy_3)),\ \ \ 0\le j \le 2.
\ee
Moreover,
$$
\left|\frac{\partial G}{\partial x_j}\right| \le C_3,\ \ \ 0\le j\le 3,\ \ \ \mbox{on}\ \ \Nk_1(V_{R_0}).
$$
\end{lemma}

\begin{proof}
Writing $z_0 = x_0 + iy_0$ and $\theta = \alpha + i \beta$ we obtain from (\ref{equs}) the system
of equations:
\be \label{eq-vector2}
\begin{array}{l} \alpha x_0 - \beta y_0 = x_1 \\ (\alpha^2-\beta^2) x_0 - 2 \alpha \beta y_0 = x_2 \\
(\alpha^3-3\alpha\beta^2)x_0 - (3\alpha^2\beta - \beta^3) y_0 = x_3 \end{array}
\ee
Eliminating $y_0$ from the first equation results in a $2\times 2$ linear system for $|\theta|^2$ and $\alpha$, yielding
\be \label{sol1}
|\theta|^2 = |\theta|^2(\bx)=\frac{x_2^2-x_1 x_3}{x_1^2-x_0 x_2}\ \ \mbox{and}\ \  \alpha = \alpha(\bx)=\frac{x_1x_2-x_0x_3}{2(x_1^2 - x_0 x_2)}.
\ee
We then have
\be \label{sol2}
\beta = \beta(\bx)=\sqrt{|\theta|^2(\bx) - \alpha^2(\bx)}\ \ \ \mbox{and}\ \ \
 y_0 = y_0(\bx)=\frac{\alpha(\bx) x_0-x_1}{\beta(\bx)}\,.\ee
Finally,
$
F(\bx) = \alpha(\bx) + i\beta(\bx)$ and
\begin{eqnarray} G(\bx) = y_3 & = & \Im(\theta^3(x_0+y_0)) \nonumber \\
& = &
(3\alpha^2(\bx)\beta(\bx) - \beta^3(\bx))x_0 + (\alpha^3(\bx)-3\alpha(\bx)\beta^2(\bx))y_0(\bx).\label{sol3}
\end{eqnarray}
Writing $\theta$ and $z_0$ in polar coordinates, we obtain
\be \label{polar}
x_1^2-x_0x_2 = |z_0|^2 \beta^2\ge |z_0|^2 \eta^2,
\ee
showing that the denominators in (\ref{sol1}) are bounded away from zero in $V_{R_0}$. On the other hand,
$$
\|\bx\|_\infty \le |z_0|b_2^3,\ \ \ \mbox{for}\ \ \bx\in V_{R_0}.
$$
Assuming that $|\Delta x_i|\le 1$, we have
$$
|(x_1+\Delta x_1)^2 - (x_0 + \Delta x_0) (x_2 + \Delta x_2) - (x_1^2-x_0 x_2)| \le 4\|\bx\|_\infty + 2
= O_{b_2}(|z_0|).
$$
Together with (\ref{polar}), this shows that the denominators in (\ref{sol1}) are bounded away from zero in $\Nk_1(V_{R_0})$, and so $F(\bx)=\theta(\bx)$ is well-defined
  in the neighborhood $\Nk_1(V_{R_0})$, for $R_0$ sufficiently large.
We then have
$$
(x_1^2-x_0x_2)^{-1} = O_{\eta}(|z_0|^{-2})\ \ \ \mbox{and}\ \ \ \|\bx\|_\infty = O_{b_2}(|z_0|)
$$
in the whole neighborhood $\Nk_1(V_{R_0})$.
Therefore, (\ref{sol1}) implies
$$
\left|\frac{\partial |\theta|^2}{\partial x_j}\right| = O_{b_2,\eta}(|z_0|^{-1}) = O_{b_2,\eta}(R_0^{-1})\ \ \
\mbox{and}\ \ \
\left|\frac{\partial \alpha}{\partial x_j}\right| = O_{b_2,\eta}(R_0^{-1}).
$$
Finally, (\ref{sol2}) and (\ref{sol3}) yield
$$
\left|\frac{\partial \beta}{\partial x_j}\right| = O_{b_2,\eta}(R_0^{-1}),\ \ \
\left|\frac{\partial y_0}{\partial x_j}\right| = O_{b_2,\eta}(1),\ \ \ \mbox{and}\ \ \
\left|\frac{\partial y_3}{\partial x_j}\right| =O_{b_2,\eta}(1),
$$
all in the entire neighborhood $\Nk_1(V_{R_0})$.
The lemma is proved.
\end{proof}

Now denote
$$
Y_n = \Im(\theta^n t), \ \ \ n\in \N.
$$
Then we have from \eqref{eq:K+eps} and Lemma~\ref{lem-tech} for $n\ge N_0=N_0(b_1,b_2,\eta)$:
$$
Y_{n+3} = G (K_n + \eps_n,\ldots, K_{n+3}+\eps_{n+3}).
$$
Let
\be \label{def-G}
\wtil{Y}_{n+3} = G(K_n,\ldots,K_{n+3}),\ \ \ n\ge N_0,
\ee
which is well-defined, since $|\eps_k|\le 1/2$.
From  Lemma~\ref{lem-tech} it also follows that
\be \label{eq2}
|Y_{n} - \wtil{Y}_{n}| \le 4 C_3 \max\{|\eps_{n-3}|,\ldots,|\eps_{n}|\},\ \ \ n\ge N_0+3=:N_1
\ee
(the factor $4$ comes from the estimate $\|\bx\|_1 \le 4\|\bx\|_\infty$).

\begin{lemma} \label{lem-theta}
There exist $C_4>0$ and $N_2\in \N$, which depend only on $b_1,b_2,\eta$, such that
$$
\left|\theta - \frac{K_{n+1} + i \wtil{Y}_{n+1}}{K_n + i\wtil{Y}_n}\right| \le C_4|\theta|^{-n} \max\{|\eps_{n-3}|,\ldots,|\eps_{n+1}|\}\ \ \ \mbox{for}\ n\ge N_2.
$$
\end{lemma}

\begin{proof}
First we write
\begin{eqnarray*}
\left|\theta - \frac{K_{n+1} + i Y_{n+1}}{K_n + iY_n}\right| & = & \left| \frac{K_{n+1} + \eps_{n+1} + i Y_{n+1}}{K_n +\eps_n + iY_n} - \frac{K_{n+1} + i Y_{n+1}}{K_n + iY_n}\right| \\[1.2ex]
&\le  & \frac{|\eps_{n+1}|}{|K_n +\eps_n + iY_n|} + \frac{|\eps_n|}{ |K_n +\eps_n + iY_n|}\cdot \frac{|K_{n+1} + iY_{n+1}|}{|K_n + iY_n|} \\[1.3ex]
& = & O_{b_1,b_2,\eta}(|\theta|^{-n}) \cdot\max\{|\eps_n|, |\eps_{n+1}|\},
\end{eqnarray*}
using that  $$|K_n +\eps_n + iY_n|=|\theta|^n |t|\in [|\theta|^n,|\theta|^{n+1}]$$
and
$|K_n+iY_n|\in [\half|\theta|^n, 2|\theta|^{n+1}]$
 in the last step. Then we estimate
\begin{eqnarray*}
\left| \frac{K_{n+1} + i Y_{n+1}}{K_n + iY_n} - \frac{K_{n+1} + i \wtil{Y}_{n+1}}{K_n + i\wtil{Y}_n}\right|
& \le & \frac{|\wtil{Y}_{n+1}-Y_{n+1}|}{|K_n + i \wtil{Y}_n|} + \frac{|\wtil{Y}_n-Y_n| \cdot |K_{n+1}+iY_{n+1}|}{|K_n+Y_n|\cdot|K_n + \wtil{Y}_n|} \\[1.2ex] & = & O_{b_1,b_2,\eta}(|\theta|^{-n} )\cdot\max\{|\eps_{n-3}|,\ldots,|\eps_{n+1}|\},
\end{eqnarray*}
for $n$ sufficiently large,
using (\ref{eq2}). The claim of the lemma follows.
\end{proof}

Now we continue the proof of Proposition \ref{prop:EK-combinatorial}, following the general scheme of the ``Erd\H{o}s-Kahane argument'', see e.g. \cite{PSS00, ShmerkinSolomyak14}.
By Lemma~\ref{lem-theta},
\be \label{theta-bound}
\theta \in B\bigl(\Psi(K_{N-3},\ldots,K_{N+1}), C_4b_1^{-N}\bigr),\ \ \ N\ge N_2,
\ee
where
$$
 \Psi(K_{N-3},\ldots,K_{N+1}) = \frac{K_{N+1} + iG(K_{N-2},\ldots,K_{N+1})}{K_N + iG(K_{N-3},\ldots,K_{N})}\,.
$$
Thus we need to estimate the number of possible integer sequences $(K_n)_{n\le N}$.
By Lemma~\ref{lem-theta},
$$
\left| \frac{K_{n+2} + i \wtil{Y}_{n+2}}{K_{n+1} + i\wtil{Y}_{n+1}} - \frac{K_{n+1} + i \wtil{Y}_{n+1}}{K_n + i\wtil{Y}_n}\right| \le 2C_4|\theta|^{-n} \max\{|\eps_{n-3}|,\ldots,|\eps_{n+2}|\}\ \ \ \mbox{for}\ n\ge N_2.
$$
It follows that for some $C_5 = C_5(b_1,b_2,\eta)$,
\be \label{eq3}
\left|K_{n+2} - \Re\left(\frac{(K_{n+1} + i \wtil{Y}_{n+1})^2}{K_n + i \wtil{Y}_n}\right)\right| \le C_5 \max\{|\eps_{n-3}|,\ldots,|\eps_{n+2}|\}  \ \ \mbox{for}\ n\ge N_2.
\ee
Let
$$
\rho:= (2C_5)^{-1}\ \ \ \mbox{and}\ \ \ M:= 2C_5+1.
$$
The estimate (\ref{eq3}), together with (\ref{def-G}), immediately implies the following

\begin{lemma} \label{lem-immed} Consider an arbitrary $\theta \in H_{b_1,b_2,\eta}$ and $t\in \C$, with $|t|\in [1,|\theta|]$, and define the corresponding sequences $K_n, \eps_n$ by \eqref{eq:K+eps}. Then the following holds:

{\rm (i)} If $\max\{|\eps_{n-3}|,\ldots,|\eps_{n+2}|\}\le \rho$ for $n\ge N_2$, then $K_{n+2}$ is uniquely determined by  $(K_j)_{j=n-3}^{n+1}$;

{\rm (ii)} for all $n\ge N_2$, there are at most $M$ choices for $K_{n+2}$, given $(K_j)_{j=n-3}^{n+1}$.
\end{lemma}

Now we can finish the proof of Proposition \ref{prop:EK-combinatorial}.  Assume that $N>N_2$. Fix $\theta\in \Ek_{N,\delta}$ and $t$, with $|t|\in [1,|\theta|]$. Since $|\theta|\in [b_1,b_2]$, there
are $O_{b_1,b_2,\eta}(1)$ choices for the initial part of the sequence $K_1,\ldots,K_{N_2}$ (recall that
$K_n$ is the nearest integer to $\Re(\theta^n t)$). The set $J:=\{n\in [1,N]:\ |\eps_n|>\rho\}$ has cardinality
at most $\lfloor\delta N \rfloor$, by the definition of $\Ek_{N,\delta}$. In view of Lemma~\ref{lem-immed},
given $J$, there are at most $O_{b_1,b_2,\eta}(M^{5\delta N})$ choices for the sequence $K_1,\ldots,K_N$.
Thus the total number of sequences corresponding to points in $\Ek_{N,\delta}$, hence also the number of balls of radius $O_{b_1,b_2,\eta}(b_1^{-n})$ needed to cover $\Ek_{N,\delta}$, is at most
$$
O_{b_1,b_2,\eta}(M^{5\delta N}) {N \choose \lfloor \delta N\rfloor} = \exp(O_{b_1,b_2,\eta}(\delta \log(1/\delta)N)),
$$
as desired.
\end{proof}

The proof of the second part of Theorem \ref{mainthm:E-K} is similar, but simpler. We will therefore omit some details.

\begin{proof}[Proof of Theorem \ref{mainthm:E-K}(ii)]
Again, let $\Lambda=\{1,\ldots,m\}$ with $m\ge 3$. Since replacing $a_k$ by $(a_k-a_i)/(a_j-a_i)$ (for fixed $i,j$) has the effect of applying a homothety to $\nu_{\lam,\bba}^p$, it is enough to prove that there is a set $E\subset\C$ with $\dim_H(E)=0$, such that if $a_1=0, a_2=1$ and $a_3=u\in\C\setminus E$, then $\widehat{\nu_{\lam,\bba}^p}$ has a power Fourier decay. Hence, fix $\lam\in\DmR$ and $a_4,\ldots,a_m\in\C$, and write $\eta_{u,p} = \nu_{\lam,\bba}^p$, where $\bba=(0,1,u,a_4,\ldots,a_m)$.

By the expression \eqref{eq:FT-product} of $\widehat{\eta_{u,p}}$ as an infinite product, we get
\begin{align*}
|\widehat{\eta_{u,p}}(\xi)| & \le  \prod_{n=0}^\infty \Bigl(\bigl|p_1 + p_2 e^{2\pi i \Re(\lam^n \ov{\xi})} + p_3 e^{2\pi i \Re(\lam^n u \ov{\xi})} \bigr| + (1-p_1-p_2-p_3)\Bigr) \\
& \le  \prod_{n=0}^\infty \bigl(1 - c_1\max(\|\Re(\lam^n \ov{\xi})\|,\|\Re(\lam^n u \ov{\xi})\|)^2\bigr),
\end{align*}
for some constant $c_1$ depending on $p$. As in the proof of the first part, write $\theta=\lam^{-1}$, and  let $\xi = \ov{t \lam^{-N}}$ with $|t| \in [1, \theta]$. Then we have
\be \label{eq:EKdecay2}
|\widehat{\eta_{u,p}}(\xi)| \le  \prod_{n=1}^N \left(1-c_1\max(\|\Re(\theta^n t)\|,\|\Re(\theta^n u t)\|)^2\right).
\ee
Therefore, the task is to show that $\max(\|\Re(\theta^n t)\|,\|\Re(\theta^n u t)\|)$ is bounded away from zero
for $n$ in a subset of positive lower density, uniformly in $t$, such that $|t|\in [1,|\theta|]$.

It is enough to prove the claim in the region $H_r = \{ u\in\C: r^{-1}\le |u|\le r\}$ for each $r>1$.  As in the proof of part (i) of the theorem, this is a consequence of a combinatorial proposition (that we state and prove next); as the deduction is essentially identical, we omit it.
\end{proof}

\begin{prop} \label{prop:EK-combinatorial2}
There is a constant $\rho=\rho(r,\lam)>0$ such that
\[
\left\{ u\in H_{r}:\, \frac{1}{N} \min_{|t|\in [1,|\theta|]} \#\{n\in \{1,\ldots,N\}:\, \max(\|\Re(\theta^n t)\|,\|\Re(\theta^n u t)\|)\le \rho\} > 1- \delta\right\}
\]
can be covered by $\exp(C_2\delta\log(1/\delta) N)$ balls of radius $|\lambda|^{N}$, where $C_2=C_2(r)$ is independent of $N$.
\end{prop}
\begin{proof}
Write
\begin{align*}
\Re(\theta^n t) &= K_n+\eps_n = x_n,\\
\Re(\theta^n u t) &= L_n+\delta_n = x'_n,
\end{align*}
with $|\e_n|,|\delta_n|<1/2$. Recall that (unlike part (i)) $\lam=\theta^{-1}=\alpha+i\beta$ is fixed. Write $\theta^n t = x_n+i y_n$, $\theta^n ut=x'_n+i y'_n$. A straightforward calculation shows that
\begin{align*}
y_{n+1} &= \beta^{-1}(\alpha x_{n+1}-x_n),\\
y'_{n+1} &= \beta^{-1}(\alpha x'_{n+1}-x'_n),
\end{align*}
whence
\[
u = \frac{x'_{n+1}+i y'_{n+1}}{x_{n+1} + i y_{n+1}} = \frac{x'_{n+1}+i\beta^{-1}(\alpha x'_{n+1}-x'_n) }{x_{n+1}+ i \beta^{-1}(\alpha x_{n+1}-x_n)}.
\]
From here another calculation similar to, but easier than, the one in Lemma \ref{lem-theta} yields that
\be \label{eq:EK-cover}
\left| u- \frac{L_{n+1}+i\beta^{-1}(\alpha L_{n+1}-L_n)}{K_{n+1}+i\beta^{-1}(\alpha K_{n+1}-K_n)}\right| \le C_2 |\theta|^{-n}
\ee
where $C_2>0$ depends on $\theta$ and $r$ only. Hence we are left to count the number of possible sequences $(K_i,L_i)_{i=1}^{n+2}$ for which $u$ is in the set in the statement of the lemma.

Also, we have
\[
x_{n+2} =\Re(\lam^{-1}(x_{n+1}+i y_{n+1})) =  |\theta|^2(\alpha x_{n+1}+\beta y_{n+1})= |\theta|^2(2\alpha x_{n+1}-x_n),
\]
and likewise for $x'_{n+2}$, which imply that
\begin{align*}
|K_{n+2} - |\theta|^2(2\alpha K_{n+1}-K_n)| &\le C_3\max(|\eps_n|,|\eps_{n+1}|),\\
|L_{n+2} - |\theta|^2(2\alpha L_{n+1}-L_n)| &\le C_3\max(|\delta_n|,|\delta_{n+1}|).
\end{align*}
Let $\rho:=(2C_3)^{-1}$, $M:=(2C_3+1)^2$. Similarly to Lemma \ref{lem-immed}, we see that, fixing $u\in H_r$ and $t\in\C$ with $|t|\in [1,\theta]$,
\begin{enumerate}
\item[(a)] If $\max(|\eps_n|,|\eps_{n+1}|,|\delta_n|,|\delta_{n+1}|)<\rho$, then $(K_{n+2},L_{n+2})$ is uniquely determined by $K_{n},L_{n},K_{n+1},L_{n+1}$.
\item[(b)] Given $K_{n},L_{n},K_{n+1},L_{n+1}$, there are at most $M$ possibilities for $(K_{n+2},L_{n+2})$.
\end{enumerate}
From here, we can  count the number of possible sequences $(K_i,L_i)_{i=1}^{n+2}$ (corresponding to points $u$ in the set in question), using an argument nearly identical to that used to finish the proof of Proposition \ref{prop:EK-combinatorial}. Together with \eqref{eq:EK-cover}, this concludes the proof.
\end{proof}

Finally, we give the short proof of the last claim of Theorem \ref{mainthm:E-K}, using the original argument of Erd\H{o}s \cite{Erdos40} and Kahane \cite{Kahane71}.
\begin{proof}[Proof of Theorem \ref{mainthm:E-K}(iii)]
  Fix $k\in\N$ and $\e,\eta>0$. Let $H:=H_{2,4,\eta/5}$ (the choice of $2$ and $4$ is arbitrary; $5$ is a sufficiently large constant).  By the first part of Theorem \ref{mainthm:E-K} there are $C,\gamma>0$ such that $|\widehat{\nu_{\lam,\bba}^p}(\xi)| \le C\, |\xi|^{-\gamma}$ whenever $\lam^{-1}\in H\setminus \mathcal{E}$, where $\dim_H(\mathcal{E})<\e$. For any $\ell\in\N$, we have a decomposition
\[
\nu_{\lam,\bba}^p = \nu_{\lam^\ell,\bba}^p * S_\lam\nu_{\lam^\ell,\bba}^p * \cdots * S_{\lam^{\ell-1}} \nu_{\lam^\ell,\bba}^p,
\]
where we recall that $S_u(z)=u z$. This well-known fact can be seen e.g.\ from expressing $\nu_{\lam,\bba}^p$ as an infinite convolution. Hence, if $\lam^{-\ell} \in H\setminus \mathcal{E}$, then
\[
|\widehat{\nu_{\lam,\bba}^p}(\xi)| \le C(\lam,\ell)\,|\xi|^{-\ell \gamma}.
\]
In particular, if $\ell\gamma>k+2$, then $\nu_{\lam,\bba}^p$ has a density in $C^k$ (see e.g. \cite[Proposition 3.2.12]{Grafakos09}). Pick $\ell_0\in\N$ such that $\ell_0\gamma>k+2$ and $\bigcup_{\ell=\ell_0+1}^\infty (4^{-1/\ell},3^{-1/\ell})=:(1-\delta,1)$ for some $\delta\in (0,1/3)$. Finally, set
\[
E = \bigcup_{\ell\ge \ell_0} \{ \lam:\lam^{-\ell}\in \mathcal{E}\}.
\]
Suppose $|\lam|\in (1-\delta,1)$ and $|\Im(\lam)|>\eta$. Then $\lam^{-\ell}\in [2,4]$ for two consecutive values of $\ell$. A short calculation shows that, for one of these two values, we also have $\Im(\lam^{-\ell})>\eta$, so $\lam^{-\ell}\in H$. The claim is now clear.
\end{proof}

\section{Absolute continuity: Proofs of Theorems \ref{mainthm:complex-SSM} and \ref{mainthm:convolutions}}

\subsection{Hochman's results on super-exponential concentration}
\label{subsec:hochman}
Here we recall a recent result of Hochman that will play a central role in the proof of our main theorems. We state only the special case we will require.
It is well known that $\dim \nu_{\lam,\mathbf{a}}^p\le \min(s(\lam,p),2)$ and equality is expected to ``typically'' hold. Hochman's results provide some very weak conditions under which equality indeed does hold. Given an IFS $( z\mapsto \lambda z + a_i )_{i\in\Lambda}$ with $\lam\in\D$, $a_i\in\C$,  let
\[
\Delta_n(\lam,\textbf{a}) = \min_{i\neq j\in\Lambda^n} \left|\sum_{k=0}^{n-1} \lambda^k a_{i_{k+1}} - \sum_{k=0}^{n-1} \lambda^k a_{j_{k+1}}\right|.
\]
This represents the closest distance between $n$-level cylinders coming from different words. It is easy to see that $\Delta_n(\lam,\mathbf{a})$ is decreasing, tends to $0$ at least exponentially fast, and $\Delta_n(\lam,\mathbf{a})=0$ for some $n$ if and only if there is an exact overlap. Hochman's Theorem asserts that, provided $\lam$ is non-real, there is no dimension drop unless the convergence of $\Delta_n(\lam,\mathbf{a})$ to zero is super-exponential:
\begin{thm} \label{thm:hochman}
Let $( z\mapsto \lambda z + a_i )_{i\in\Lambda}$ be an IFS as above, and let $(p_i)_{i\in\Lambda}$ be a probability vector. Then one of the following three alternatives hold:
\begin{enumerate}[(i)]
\item $\dim\nu_{\lam,\mathbf{a}}^p = \min\{2,s(\lam,p)\}$,
\item $\Delta_n(\lam,\mathbf{a})\to 0$ super-exponentially (i.e. $\log \Delta_n(\lam,\mathbf{a})/n\to- \infty$).
\item $\lambda\in\R$.
\end{enumerate}
\end{thm}
When (ii) holds, we say that there is \emph{super-exponential concentration of cylinders}. This theorem is a special case of \cite[Theorem 1.5]{Hochman15}. Using this result, Hochman proved that in very general parametrized families of self-similar measures, Hausdorff and similarity dimensions coincide outside of a set of packing dimension $\ell-1$, where $\ell$ is the dimension of the parameter space, see \cite[Theorem 1.10]{Hochman15}. This is not enough for our purposes, so we appeal to arguments specific to our situation.

\begin{prop} \label{prop:zero-dim-scc}
Let $\mathbf{a} = (a_i)_{i\in\Lambda}\subset \C^\Lambda$, $\#\Lambda\ge 2$ with all $a_i$ different. Then
\be \label{eq-dim1}
\dim_P\left(\{ \lam\in\D: \log\Delta_n(\lam,\mathbf{a})/n\to-\infty\}\right)=0,
\ee
where $\dim_P$ denotes packing dimension. Thus,
\be
\dim_P\left( \left\{\lam\in \D\setminus \R:\ \exists\,p\in\PP_\Lambda,\ \dim(\nu_{\lam,\mathbf{a}}^p) < \min\{2, s(\lam,p)\}\right\}\right) = 0.
\ee
\end{prop}
\begin{proof} The second statement follows from the first and Theorem \ref{thm:hochman}; here we use that $\lam\notin\R$. It suffices to show (\ref{eq-dim1}) with $\D$ replaced by
$$
\mathbb{A}_{\rho,r}:= \{\lam:\ \rho< |\lam|<r\}
$$
for any fixed $0 < \rho< r<1$.

Write $\Ak=\{ a_i:i\in\Lambda\}$. For  $u,v\in \Ak^n$, let
$$
\Delta_{u,v}(\lam) :=  \sum_{j=0}^{n-1} (u_j - v_j) \lam^j.
$$
We consider $\Delta_{u,v}$ as functions $\mathbb{A}_{r,\rho}\to \C$. Note that $\Delta_n(\lam,\mathbf{a})=\min_{i\neq j\in\Ak^n} |\Delta_{u,v}(\lam)|$. Hence, it suffices to show that $\dim_P(E)=0$, where
$$
E = \bigcap_{\eps>0} \bigcup_{N=1}^\infty \bigcap_{n> N} \bigcup_{u\ne v\in \Ak^n} \Delta_{u,v}^{-1}(B_{\eps^n}(0)).
$$
Let $\Dk = \Ak-\Ak$. Clearly, $\Delta_{u,v}(\lam)$ is a polynomial of degree $\le n-1$ in $\lam$, with
coefficients in $\Dk$.

\begin{lemma} \label{lem1}
Let $\Dk$ be a finite subset of $\C$, with $0\in \Dk$, and $d_* = \min(|a|:\,a\in \Dk\setminus \{0\})$.
Then for every $r<1$ there exists $k_r \ge 1$ such that for any polynomial $p$ with coefficients in $\Dk$ of degree $\le n$, any $\rho>0$, and any $\eps>0$, the set $\{\lam\in \mathbb{A}_{\rho,r}:\ |p(\lam)| < \eps^n\}$ may be covered by $k_r$ disks of
radius $d_*^{-1/k_r} (\frac{2\eps}{\rho(1-r)})^{n/k_r}$.
\end{lemma}

We deduce the proposition first.
It follows from the lemma that the set
$$
E_{\eps,n}:= \bigcup_{u\ne v \in \Ak^n} \Delta_{u,v}^{-1}(B_{\eps^n}(0))
$$
may be covered by $k_r (\# \Ak)^{2n}$ disks of diameter $C(\frac{2\eps}{\rho(1-r)})^{n/k_r}$, hence
$$
\dim_P\left(\bigcap_{n>N} E_{\eps,n}\right) \le \ov{\dim}_B\left(\bigcap_{n>N} E_{\eps,n}\right) \le \frac{2k_r\log(\#\Ak) }{-\log(\frac{2\eps}{\rho(1-r)})}
$$
for all $N$, where $\ov{\dim}_B$ denotes upper box-counting (or Minkowski) dimension. Therefore,
$$
\dim_P\left(\bigcup_{N=1}^\infty \bigcap_{n> N}  \bigcup_{u\ne v\in \Ak^n} \Delta_{u,v}^{-1}(B_{\eps^n}(0))\right) \le \frac{2k_r\log(\#\Ak) }{-\log(\frac{2\eps}{\rho(1-r)})}\,,
$$
and since the latter tends to zero as $\eps\to 0$, the desired claim follows.
\end{proof}

\medskip

\begin{proof}[Proof of Lemma~\ref{lem1}]
Let
$$
\Bk_\Dk:= \Bigl\{\sum_{j=0}^\infty a_j z^j:\ a_j \in \Dk,\ a_0\ne 0 \Bigr\}
$$
be the set of power series with coefficients in $\Dk$, non-vanishing at zero.
Fix $r<1$.
Observe that $\Bk_\Dk$ is a normal family in the open unit disk $\D$, hence it is compact on any compact subset of $\D$.  Therefore, there exists $k_r\ge 1$ such that the number of zeros of any function from $\Bk_D$ in the closed disk $\ov{B}_{(r+1)/2}(0)$, counting with multiplicities, is at most $k_r$:
\be \label{eq:bounded-zeros}
\forall f\in \Bk_\Dk,\ \ \ \ \ \#\{z:\ |z| \le (r+1)/2,\ f(z)=0\}\le k_r.
\ee
Indeed, otherwise a subsequential limit in $\Bk_\Dk$ (which is not constant zero by the definition of $\Bk_\Dk$) would have infinitely many zeros in $\ov{B}_{(r+1)/2}(0)$. (In fact, an explicit estimate for $k_r$ in terms of coefficient bounds is given in \cite[Theorem 2]{BBBP}, but for us this is unimportant.)

Let $p$ be a polynomial of degree $\le n$ with coefficients in $\Dk$. Then we have $p(z) = z^s q(z)$, where
$q\in \Bk_\Dk$ does not vanish at $0$. We can write
$$
q(z) = a(z-z_1)\cdots (z-z_\ell),
$$
where $a\in \Dk$ and $z_1,\ldots,z_\ell\ne 0$ are all the zeros of $q$, counted with multiplicities. Thus,
$$
|p(z)| \ge d_*\cdot |z|^s \cdot\prod_{|z_j|\le \frac{1+r}{2}} |z-z_j|\cdot\prod_{|z_j|> \frac{1+r}{2}} |z-z_j|.
$$
Therefore,
\begin{eqnarray*}
\lam\in \mathbb{A}_{\rho,r},\ |p(\lam)| \le \eps^n\ \Longrightarrow\ \prod_{|z_j|\le \frac{1+r}{2}} |\lam-z_j|  & \le &
d_*^{-1} |\lam|^{-s} \prod_{|z_j|> \frac{1+r}{2}} |\lam-z_j|^{-1}\cdot\eps^n \\
& \le & d_*^{-1} \rho^{-n} \left(\frac{1-r}{2}\right)^{-n}\eps^{n}.
\end{eqnarray*}
In view of \eqref{eq:bounded-zeros}, $\lam\in \mathbb{A}_{\rho,r},\ |p(\lam)| \le \eps^n$, implies
$$
\min\{|\lam-z_j|:\ |z_j| \le (1+r)/2\} \le d_*^{-1/k_r}\left(\frac{2\eps}{\rho(1-r)}\right)^{n/k_r},
$$
and the claim of the lemma follows.
\end{proof}

We have a similar, but easier, result in the setting of Theorem \ref{mainthm:convolutions}.
\begin{prop} \label{prop:zero-dim-scc-conv}
Fix $\lam\in\DmR$ and $\mathbf{a} = (a_i)_{i\in\Lambda}\subset \C^\Lambda$,\ $\#\Lambda\ge 2$, such that $(z\mapsto \lambda z+a_i)_{i\in\Lambda}$ has no super-exponential concentration of cylinders. Then
\be \label{eq:dimP-translations}
\dim_P\left( \left\{u\in\C:\ \exists\,p\in\PP_\Lambda,\ \dim\left(\nu_{\lam,\mathbf{a}}^p*S_u\nu_{\lam,\mathbf{a}}^p\right) < \min\{2, 2 s(\lam,p)\}\right\}\right) = 0.
\ee
\end{prop}
\begin{proof}
The measure $\nu_{\lam,\mathbf{a}}^p*S_u\nu_{\lam,\mathbf{a}}^p$ is the attractor of the IFS $( \lambda z+ b_{ij}(u))_{{ij}\in\Lambda^2}$ with weight $q\in\PP_{\Lambda^2}$, where $b_{ij}(u) = a_i+u a_j$ and $q_{ij}=p_i p_j$. Write $\Ak=\{a_i:i\in\Lambda\}$.

Notice that
\[
\Delta_n(\lam,\mathbf{b})=\min_{v,v',w,w'\in\Ak^n:(v,w)\neq (v',w')} \left|  \sum_{i=0}^{n-1} (v_i-v'_i) \lam^i + u (w_i-w'_i) \lam^i \right|.
\]
Denote $\Dk=\Ak-\Ak$. We can rewrite the above as
\[
\Delta_n(\lam,\mathbf{b})= \min_{p,q\in\mathcal{Q}_n: p\neq 0 \text{ or }q\neq 0}  |p(\lam)+u q(\lam)|,
\]
where $\mathcal{Q}_n$ is the family of polynomials with coefficients in $\Dk$ of degree at most $n-1$.  Hence, in light of Theorem \ref{thm:hochman}, it suffices to show that $\dim_P(E)=0$, where
\[
E = \bigcap_{\eps>0} \bigcup_{N=1}^\infty \left( \bigcap_{n> N}  \bigcup_{p,q\in\mathcal{Q}_n, p\neq 0 \text{ or }q\neq 0} \{ u: |p(\lam)+u q(\lam)|<\e^n\}\right) =: \bigcap_{\e>0}\bigcup_{N=1}^\infty E_{\e,N}.
\]
Note that
\[
\Delta_n(\lam,\mathbf{a}) = \min_{p\in\mathcal{Q}_n, p \neq 0} |p(\lam)|.
\]
Since, by assumption, $\Delta_n(\lam,\mathbf{a})$ does not have super-exponential decay, there exists $c>0$ such that $|p(\lam)|\ge c^n$ for each nonzero $p\in\mathcal{Q}_n$, $n\in \N$. If $\e<c$, then $\{u: |p(\lam)+uq(\lam)|<\e^n\}$ is empty unless $q\neq 0$ and so, for any $n>N$,
\[
E_{\e,N} \subset\bigcup_{p,q\in\mathcal{Q}_n, q\neq 0} \{ u: |p(\lam)+u q(\lam)|<\e^n\} = \bigcup_{p,q\in\mathcal{Q}_n, q\neq 0} B_{\tfrac{\e^n}{|q(\lam)|}}(-\tfrac{p(\lam)}{q(\lam)}).
\]
This shows that $E_{\e,N}$ can be covered by $(\#\mathcal{Q}_n)^2\le (\#\Ak)^{4n}$ balls of radius $(\e/c)^n$. Since $n>N$ is arbitrary, this shows that
\[
\dim_P(E_{\e,N}) \le \ov{\dim}_B(E_{\e,N}) \le \frac{4\log|\Ak|}{-\log(\e/c)} \to 0 \text{ as }\e\to 0.
\]
This implies that $\dim_P(E)=0$, as desired.
\end{proof}
In particular, the proposition holds under the OSC, since this is well known to imply no super-exponential concentration of cylinders, see e.g. \cite[Proposition 1]{BandtGraf92}.

\subsection{Proofs of main results}

We can now complete the proofs of Theorems \ref{mainthm:complex-SSM} and \ref{mainthm:convolutions},  following closely the arguments of \cite{Shmerkin13} and \cite{ShmerkinSolomyak14}.

\begin{proof}[Proof of Theorem \ref{mainthm:complex-SSM}]
Let us decompose the supercritical parameter space
\[
P = \{ (\lam,p)\in \DmR\times\PP_\Lambda, h(p) > 2|\log\lam|\}
\]
into countably many pieces $P_k$, $k\ge 3$, such that
\[
\frac{h(p)}{|\log\lam|} > 2 + 4/k \quad\text{for all } (\lam,p)\in P_k.
\]
It is enough to show that for each $k$ there is an exceptional set $E_k$ of Hausdorff dimension $0$ such that if $(\lam,p)\in P_k\setminus E_k$, then $\nu_{\lam,\mathbf{a}}^p$ is absolutely continuous with a density in $L^q$ for some $q>1$. The theorem will then follow with $E=\bigcup_{k=1}^\infty E_k$. Hence, we fix $k$ for the rest of the proof.

Recall that $\nu_{\lam,\mathbf{a}}^p$ can be characterized as the distribution of the random sum $\sum_{n=0}^\infty \lambda^n X_n$, where the $X_n$ are i.i.d. with $\PP(X_n=a_i)=p_i$. It follows from this characterization that
\be \label{eq:conv}
\nu_{\lam,\mathbf{a}}^p = \mu_{\lam,\mathbf{a}}^p * \eta_{\lam,\mathbf{a}}^p,
\ee
where  $\eta_{\lam,\mathbf{a}}^p$ is the distribution of $\sum_{n:k\mid n} \lambda^n X_n$, while the measure $\mu_{\lam,\mathbf{a}}^p$ is the distribution of $\sum_{n:k\nmid n} \lambda^n X_n$. These measures are again self-similar; in fact, $\eta_{\lam,\mathbf{a}}^p=\nu_{\lam^k,\mathbf{a}}^p$, and $\mu_{\lam,\mathbf{a}}^p=\nu_{\lam^k, \mathbf{b}(\lam)}^q$, where $\mathbf{b}=(b_i)_{i\in\Lambda^{k-1}}$, $q\in\PP_{\Lambda^{k-1}}$ are given by
\begin{align*}
b_{i_1\ldots i_{k-1}} &=\sum_{j=0}^{k-2}  \lambda^j a_{i_{j+1}},\\
q_{i_1\ldots i_{k-1}} &= p_{i_1}\cdots p_{i_{k-1}}.
\end{align*}
Note that $\mathbf{b}$ depends on $\lam$. In particular, if $(\lam,p)\in P_k$, then
\[
s(\lam^k,q)=\frac{h(q)}{|\log\lam^k|}=\frac{(k-1)h(p)}{k|\log\lam|} = \left(1-\frac1k\right) s(\lam,p)>2.
\]
Let
\[
E'=\{ \lam\in\D: \Delta_n(\lam,\mathbf{a}) > c^n \text{ for some } c>0\}
\]
be the set of parameters for which there is no super-exponential concentration of cylinders. We know from Proposition \ref{prop:zero-dim-scc} that $\dim_H(E')=\dim_P(E')=0$. We claim that $\Delta_n(\lam^k,\mathbf{b})\ge \Delta_{kn}(\lam,\mathbf{a})$ for all $n$. This is because the IFS $(\lambda^k z + b_i)_{i\in\Lambda^{k-1}}$ is obtained by iterating $(\lambda z+a_i)_{i\in\Lambda}$ $k$ times and then deleting some maps. More precisely, for any pair $i,j\in (\Lambda^{k-1})^n$ we have
\[
\sum_{\ell=0}^{n-1} \lambda^{k\ell} b_{i_{\ell+1}} - \sum_{\ell=0}^{n-1} \lambda^{k\ell} b_{j_{\ell+1}} = \sum_{\ell=0}^{kn-1} \lambda^{\ell} a_{i'_{\ell+1}} - \sum_{\ell=0}^{kn-1} \lambda^\ell a_{j'_{\ell+1}},
\]
where $i'=(i_1 0 i_2 0 \cdots i_n 0)$ and likewise for $j'$, which gives the claim. It now follows that $\Delta_n(\lam^k,\mathbf{b})\ge (c^k)^n$ for all $n$ and $\lam\in\D\setminus E'$, and we deduce from Theorem \ref{thm:hochman} that if $\lam\in \D\setminus (E'\cup\R)$, then $\dim\mu_{\lam,\mathbf{a}}^p=2$.

We also know from Theorem \ref{mainthm:E-K}(i) that there is another set $E''$ of zero Hausdorff dimension, such that $|\widehat{\eta_{\lam,\mathbf{a}}^p}(\xi)|\le C\,|\xi|^{-\gamma}$ for $\lam\in\D\setminus E''$ and some $C,\gamma>0$ depending on $\lam$. We can now conclude from \eqref{eq:conv} and \cite[Corollary 5.5]{ShmerkinSolomyak14} that if $\lam\in P_k\setminus (E'\cup E'')$, then $\nu_{\lam,\mathbf{a}}^p$ is absolutely continuous with a density in $L^q$ for some $q>1$. This completes the proof.
\end{proof}

\begin{proof}[Proof of Theorem \ref{mainthm:convolutions}]
Once again, it is enough to prove that there is a set $E$ of zero Hausdorff dimension, such that the claim holds for all $u\in\C$ and $p\in P$, where
\[
P = \{ p\in\PP_\Lambda, s(\lam,p)>1+1/k\}.
\]
Recall the decomposition \eqref{eq:conv}. Since convolution is linear and commutative, we obtain
\be \label{eq:conv2}
\nu_{\lam,\mathbf{a}}^p * S_u \nu_{\lam,\mathbf{a}}^p = (\mu_{\lam,\mathbf{a}}^p * S_u \mu_{\lam,\mathbf{a}}^p ) * (\eta_{\lam,\mathbf{a}}^p * S_u \eta_{\lam,\mathbf{a}}^p).
\ee
Recall also from the proof of Theorem \ref{mainthm:complex-SSM} that $\mu_{\lam,\mathbf{a}}^p=\nu_{\lam^k,\mathbf{b}}^q$, where $s(\lam^k,q)=(1-1/k)s(\lam,p)$, and that this IFS has no super-exponential concentration of cylinders (since we assume this for $(\lambda z+a_i)_{i\in\Lambda})$. We can then apply Proposition \ref{prop:zero-dim-scc-conv} to obtain a set $E'\subset\C$ with $\dim_H(E')=0$, such that
\[
\dim(\mu_{\lam,\mathbf{a}}^p * S_u \mu_{\lam,\mathbf{a}}^p )=2 \quad\text{for all }u\in\C\setminus E',p\in P.
\]

On the other hand, since $\eta_{\lam,\mathbf{a}}=\nu_{\lam^k,\mathbf{a}}$, we have
\[
\eta_{\lam,\mathbf{a}}^p * S_u \eta_{\lam,\mathbf{a}}^p = \nu_{\lam,\mathbf{b}}^q,
\]
where $\mathbf{b}=(b_{ij}:= a_i + u a_j)_{ij\in\Lambda^2}$ and $q=(q_{ij}:=p_i p_j)_{ij\in\Lambda^2}$. Since $\#\Lambda\ge 2$ and all the $a_i$ are different, there are $i,j\in\Lambda$ such that
\[
-u=\frac{b_{i i}-b_{i j}}{b_{jj}-b_{ij}}.
\]
We can then apply Theorem \ref{mainthm:E-K}(ii) to obtain a set $E''\subset\C$ with $\dim_H(E'')=0$, such that the Fourier transform of $\eta_{\lam,\mathbf{a}}^p * S_u \eta_{\lam,\mathbf{a}}^p$ has power decay at infinity for every $u\in\C\setminus E''$.

Taking $E=E'\cup E''$, the proof is finished by virtue of  \eqref{eq:conv2} and \cite[Corollary 5.5]{ShmerkinSolomyak14}.

\end{proof}

\section{Further results}

\label{sec:generalizations}

\subsection{Fat Sierpi\'{n}ski gaskets}
\label{subsec:gasket}

Given $\lam\in (0,1$), the generalized Sierpi\'{n}ski gasket $A_\lam\subset\C$ is defined by the self-similarity relation
\[
A_\lam = \lam (A_\lam+a_1) \cup \lam(A_\lam+a_2) \cup \lam(A_\lam+a_3),
\]
where $a_i$ are the vertices of an equilateral triangle centered at the origin (if $a_i$ are arbitrary non-collinear points, one gets the same set up to an affine bijection). When $\lam\le 1/2$ the open set condition is satisfied ($\lam=1/2$ corresponds to the classical gasket), and several authors \cite{BMS04, Jordan05, JordanPollicott06, Hochman15} have studied the dimension, measure, and topology of $A_\lam$ for $\lam\in (1/2,1)$ (the ``fat'' regime). In particular, Hochman \cite[Theorem 1.16]{Hochman15} showed that
\[
\dim_H A_\lam = \min(\log 3/|\log\lam|,2)
\]
for $\lam$ outside of a set of zero Hausdorff and packing dimension. Jordan and Pollicott \cite[Theorem 2 and Example 1]{JordanPollicott06} proved that for a.e. $\lam$ in a certain interval of the supercritical region, $A_\lam$ has positive Lebesgue measure. This interval is different from the interval $(0.647...,1)$, for which Broomhead, Montaldi and Sidorov showed that $A_\lam$ has nonenmpty interior.

As a corollary of Theorem \ref{mainthm:complex-SSM}, we obtain:
\begin{cor}
Let $\nu_\lam = \nu_{\lam,\mathbf{a}}^p$ with $p=(1/3,1/3,1/3)$ be the natural self-similar measure on $A_\lam$. Then there is a set $E\subset (1/\sqrt{3},1)$ of zero Hausdorff dimension, such that $\nu_\lam$ is absolutely continuous and has a density in $L^q$ for some $q=q(\lam)>1$, for all $\lam\in (1/\sqrt{3},1)\setminus E$.

In particular, $\mathcal{L}(A_\lam)>0$ for all $\lam\in (1/\sqrt{3},1)\setminus E$.
\end{cor}
\begin{proof}
We cannot apply Theorem \ref{mainthm:complex-SSM} directly, since $\lam$ is real. However, $\nu_\lam$ and $A_\lam$ are invariant under a $\pi/3$ rotation around the origin, so we also have $\nu_\lam = \nu_{\omega\lam,\mathbf{a}}^p$ where $\omega=e^{i \pi/3}$, and the claim is now immediate from Theorem \ref{mainthm:complex-SSM}
\end{proof}

We make some further remarks.
\begin{rems}
\begin{enumerate}
\item The trick of realizing $A_\lam$ as the attractor of an IFS with rotations was already used in \cite{Hochman15}. Because of the need to have this rotation, the proof only applies to the natural self-similar measure (other self-similar measures on $A_\lam$ are not invariant under any rotations).
\item It is crucial that the exceptional set in Theorem \ref{mainthm:complex-SSM} has dimension zero - this allows us to obtain the same conclusion if we restrict $\lam$ to a one-dimensional family, as in this case.
\item It follows immediately from Theorem \ref{mainthm:E-K}(iii) (using the rotated IFS) that for any $k\ge 1$,
\[
\lim_{\delta\to 0} \dim_H\{ \lam\in (1-\delta,1): \nu_\lam\notin C^k \} = 0.
\]
Recall that for sets much more is true: $A_\lam$ has nonempty interior for \emph{all} $\lam$ near (and not so near) $1$.
\end{enumerate}
\end{rems}

\subsection{More general parametrized families}

Theorems \ref{mainthm:complex-SSM} and \ref{mainthm:convolutions} both assert that in the parametrized family in question, absolute continuity holds (in the supercritical region where the similarity dimension exceeds $2$) outside of a zero-dimensional set of possible exceptions. Although the part of the argument that establishes zero dimension of exceptions for the dimension statement (i.e. Propositions \ref{prop:zero-dim-scc} and \ref{prop:zero-dim-scc-conv}) appear to be specific to these families, it is possible to obtain weaker versions valid for more general parametrized families. As an example, we have:

\begin{prop} \label{prop:parametrized}
Let $\Omega\subset\C$ be an open domain, and suppose $\psi=(\lambda, \mathbf{a}):\Omega\to (\D\setminus\R)\times \C^\Lambda$  is a non-constant analytic map such that, for every pair of different infinite sequences $i,j\in\Lambda^\N$, the map
\[
u\mapsto \sum_{n=0}^\infty \lambda^n(u)  \left(a_{i_n+1}(u)-a_{j_n+1}(u)\right)
\]
is non-constant. Then there exists a set $E\subset\Omega$ with $\dim_H(E)\le 1$ such that for all $(u,p)\in(\Omega\setminus E)\times \PP_\Lambda$ such that $s(\lam(u),p)>2$, the measure $\nu_{\lam(u),\bba(u)}^p$ is absolutely continuous with an $L^q$ density for some $q=q(u)>1$.
\end{prop}
\begin{proof}
Suppose first $\lambda(u)$ is not constant. Then, after removing critical points of $\lam$ from $\Omega$, splitting into appropriate domains and changing variable, we can assume that $\lambda(u)=u$. We are then in a situation nearly identical to Theorem \ref{mainthm:complex-SSM}, except that Proposition \ref{prop:zero-dim-scc} does not apply, but we can instead appeal to Hochman's general result on dimension of exceptions, \cite[Theorem 1.10]{Hochman15}.

If, instead, $\lambda$ is constant, then necessarily $|\Lambda|\ge 3$, for otherwise the non-degeneracy condition would fail. Indeed, if $|\Lambda|=2$ and $s(\lam,p)>2$, then there are different words $i,j\in\{0,1\}^\N$ such that $\sum_{n=1}^\infty \lam^n(a_{i_n+1}-a_{j_n+1})=0$ (otherwise, the IFS $(\lambda x, \lambda x+1)$ would satisfy the open set condition, and so $s(\lam,p)\le 2$). After changing variable in the usual way, we can then assume that $a_0\equiv 0, a_1\equiv 1$ and $a_2\equiv u$,  and continue arguing as in the proof of Theorem \ref{mainthm:convolutions}, but again using  \cite[Theorem 1.10]{Hochman15}.
\end{proof}

\subsection{Further results for self-similar sets}

So far, we had to assume that the iterated function systems we work with are homogeneous (all maps have the same linear parts); this ensures that the self-similar measures have a convolution structure, which is central to our method. However, for self-similar sets, one can use a standard approximation argument to obtain similar results also in the non-homogeneous situation. For example, we have the following consequence of Theorem \ref{mainthm:convolutions}.

\begin{cor}
Let $A\subset\C$ be any self-similar set (that is, $A=\bigcup_{i\in\Lambda} f_i(\Lambda)$ for some contracting similarities $f_i$). If $\dim_H A>1$, then
\[
\dim_H(\{ u\in \C: \mathcal{L}(A+u A)=0\})=0.
\]
\end{cor}
\begin{proof}
There exists a self-similar set $A'\subset A$ which is generated by a homogeneous IFS, satisfies the open set condition, and has dimension $>1$: this follows e.g. from \cite[Lemma 3.6]{Orponen12} and \cite[Proposition 6]{PeresShmerkin09}. Hence the claim follows by applying Theorem \ref{mainthm:convolutions} to $A'$.
\end{proof}
In the same vein, it is possible to obtain a version of Proposition \ref{prop:parametrized} for parametrized families of self-similar sets (where the generating IFSs are not necessarily homogeneous). We leave the precise formulation to the interested reader.

\subsection{Higher dimensions?}

In this paper, we have adapted the method from \cite{Shmerkin13, ShmerkinSolomyak14} to self-similar sets and measures in the plane. Unfortunately, the approach we use breaks down in higher dimensions. Indeed, we need to work with IFSs in which the linear parts of the maps are equal, in order to be able to split the measures of interest as the convolution  of a measure of full dimension, and another measure with power Fourier decay. The main issue is in establishing full dimension of the corresponding measures in dimensions $3$ and higher: because the linear parts are all equal, they act reducibly and hence the results of Hochman \cite{Hochman15} are no longer applicable.


\end{document}